\documentclass{amsproc}

\usepackage{amssymb}

\usepackage{graphicx}
\usepackage{color}

\usepackage{}

\newcommand{\N}{\mathbb{N}}
\newcommand{\R}{\mathbb{R}}

\newtheorem{theorem}{Theorem}[section]

\theoremstyle{definition}

\theoremstyle{remark}

\numberwithin{equation}{section}

\begin{document}

\title[Regularized Factorization Method]{Regularization of the Factorization Method with Applications to Inverse Scattering}


\author{Isaac Harris}
\address{Department of Mathematics, Purdue University, West Lafayette, IN 47907}
\email{harri814@purdue.edu}
\thanks{The research of I. Harris is partially supported by the NSF DMS Grant 2107891.}



\date{}

\begin{abstract}
Here we discuss a regularized version of the factorization method for positive operators acting on a Hilbert Space. The factorization method is a {\it qualitative} reconstruction method that has been used to solve many inverse shape problems. In general, qualitative methods seek to reconstruct the shape of an unknown object using little to no a priori information. The regularized factorization method presented here seeks to avoid numerical instabilities in the inversion algorithm. This allows one to recover unknown structures  in a computationally simple and analytically rigorous way. We will discuss the theory and application of the regularized factorization method to examples coming from acoustic inverse scattering. Numerical examples will also be presented using synthetic data to show the applicability of the method. 
\end{abstract}

\maketitle

\section{Introduction}
In this paper, we will discuss a regularized version of the factorization method as well as {\color{black}its} applications scattering theory. We will briefly review the theoretical framework that was developed in \cite{RegFMme}. The factorization method (see for e.g. \cite{FM-wave,Gebauer,FMheat,DOT-RtR,RtR,firstFM,kirschpp,kirschbook}) is a method used to solve inverse shape problems and  fall under the category of {\it qualitative methods}. Qualitative methods are otherwise {\color{black}referred} to as non-iterative or direct methods. In many applications it is optimal to use qualitative methods rather than applying non-linear optimization techniques for two reasons: first is that optimization methods require a priori information (to construct an initial guess) that may not be readily available such as the number of regions to be recovered, second is that these methods can be computationally expensive and highly {\color{black}ill-conditioned}. All qualitative methods seek to recover the shape of an unknown region from little a priori information by relating the support of the region to the range of the `measured' data operator. These methods where first introduced in \cite{CK} and are frequently used in non-destructive testing where one is given  measurements on the surface (or exterior) of an object and one tries to reconstruct interior structures. This has many {\color{black}applications} in the area of medical imaging and non-destructive testing in engineering.

The factorization method solves the inverse shape problem by appealing to Picard's criteria for compact operators. To this end, a range test is used to determine the support of the unknown region denoted $D$. In general, we have that 
$$ z \in D \iff \ell_z \in \text{Range} \big( A^{1/2} \big)$$ 
where $\ell_z$ is known and the positive compact operator $A$ is given by the measurements. In order to apply Picard's criteria to reconstruct the unknown region a series is computed where one divides by the sequence of eigenvalues (or singular values) of the compact operator $A$. Since this sequence {\color{black}tends} to zero (usually rapidly) this could result in numerical instabilities. Therefore, we will develop a regularization strategy for the  factorization method motivated by the {\color{black}previous} works in \cite{arens,arens2,GLSM,Harris-Rome,RegFM}. In \cite{arens,arens2} the linear sampling method was studied by appealing to the analytical techniques in the factorization method and then applying a suitable regularization strategy. Whereas in \cite{GLSM} a new qualitative method known as the generalized linear sampling method was developed and uses a specific cost-functional for the regularization scheme in applying the linear sampling method. What we present here is mainly influenced by \cite{Harris-Rome,RegFM}. Loosely speaking, we have the result that for a positive compact operator $A: X \to X^*$ where $X$ is a Hilbert space and $X^*$ is the corresponding dual-space then 
$$\ell \in\text{Range} \big( A^{1/2} \big)\iff \liminf\limits_{\alpha \to 0} \langle x_\alpha \, , Ax_\alpha \rangle_{X\times X^*} < \infty$$
where $x_\alpha$ {\color{black}(defined below)} is the regularized solution to $Ax=\ell$. The regularization scheme that is used to compute $x_\alpha$ can be taken to be any of the standard techniques i.e. Tikhonov regularization, Spectral cutoff and Landweber iteration. Here $\langle \cdot \, , \cdot \rangle_{X\times X^*}$ is the sesquilinear dual-pairing between $X$ and $X^*$. The main analytical tool one needs to prove this result is the spectral decomposition for the given positive compact operator $A: X \to X^*$.

The {\color{black}preceding} sections are organized as follows. First, we will briefly discuss the theory behind the regularized version of the factorization method for a  positive compact operator $A: X \to X^*$ where $X$ is a Hilbert space. The analysis presented here was initially studied in \cite{RegFM}. Then, we consider two inverse shape problems coming from inverse scattering. First, we will consider the problem of recovering an isotropic scatterer from far-field measurements. Lastly, we will consider the problem of recovering a sound soft scatterer from near-field measurements.

\section{Regularized Factorization Method}\label{RFM}
In this section, we will discuss the theoretical {\color{black}framework} that was developed in \cite{RegFMme} for the regularized factorization method. The analysis here generalizes the main result in \cite{Harris-Rome}. To begin, we assume that we have a given data operator denoted by $A: X \to X^*$ acting on the Hilbert space $X$ that is positive and compact. Again, we note that here we take the notation that $X^*$ denotes the dual space of $X$ as well as $\langle \cdot \, , \cdot \rangle_{X\times X^*}$ denoting the sesquilinear dual-product between $X$ and $X^*$. Furthermore, assume that there is a separable Hilbert pivoting space $H$  with dense inclusions $X \subseteq H \subseteq X^*$ i.e. a Gelfand triple of Hilbert spaces. 

Now in \cite{RegFMme} it is proven that the operator $A$ has a spectral decomposition provided that either $X$ is a complex Hilbert space or the bilinear form
$$(x,y) \longmapsto  \langle y \, , A x \rangle_{X\times X^*} \quad \text{ for any } \quad  x,y \in X$$
is symmetric. 
Under these assumptions we have that 
\begin{align}\label{Adecomp}
A x = \sum \lambda_n  ( x,x _n)_{X} \,  \ell_n \quad \text{ or } \quad A x= \sum \lambda_n  \langle x,\ell _n \rangle_{X\times X^*} \,  \ell_n
\end{align}
for any $x \in X$. Here the decreasing sequence $\lambda_n \in {\R}_{>0}$ converges to zero {\color{black}whereas} $\{x_n \}_{n \in \N}$ is an orthonormal basis of $X$ and $\{\ell_n \}_{n \in \N}$ is an orthonormal basis of $X^*$. Moreover, $\{\ell_n \}_{n \in \N}$ is the corresponding dual-basis for $\{x_n \}_{n \in \N}$ such that 
$$\langle x_m \, , \ell_n \rangle_{X\times X^*} = \delta_{mn} \quad \text{ for any } \quad n,m \in \N.$$
This gives that $\{\lambda_n ; x_n ; \ell_n \} \in \R_{>0}\times X\times X^*$ is the singular value decomposition of the compact operator $A$. Now just as in \cite{kirschipbook} we can define the regularized solution of $Ax=\ell$ to be $x_\alpha$ which is given by 
\begin{align}\label{regsolu}
x_\alpha =  \sum \frac{\phi(\lambda_n ; \alpha)}{\lambda_n} \overline{ \langle x_n,\ell  \rangle}_{X\times X^*} \, x_n.
\end{align}
The real-valued function $\phi(t ; \alpha)$ denotes the filter associated with a given  regularization technique. Here we will assume that $\phi(t ; \alpha): \big(0 ,\lambda_1 \big] \to \R_{\geq0}$ satisfies that for all $0 < t \leq  \lambda_1$
$$\lim\limits_{\alpha \to 0} \phi(t ; \alpha) =1 \quad \text{ and } \quad  \phi(t ; \alpha) \leq C_{\text{reg}} \quad \text{for all} \,\, \alpha>0.$$

Now, {\color{black}due} to the fact that $A$ is positive and compact we have that there is a bounded linear `square root' operator  denoted $Q: X \to H$ such that $A=Q^* Q$ where the adjoint $Q^*$ is defined by  
$$(Q x , h )_{H} = \langle x , Q^* h \rangle_{X\times X^*} \quad \text{ for all } \quad  h \in {H} \textrm{ and } x \in X$$ 
see Theorem 2.2 of \cite{RegFMme} for details. We also obtain that 
$$\ell \in \text{Range}(Q^*) \iff  \sum \frac{1}{\lambda_n}  \left|\langle x_n ,\ell \rangle_{X\times X^*} \right|^2 <\infty$$
from Theorem  2.2 of \cite{RegFMme} which is proven in a similar {\color{black}as was} Picard's Criteria (see for e.g. Theorem 1.28 of \cite{TE-book}).
Then by appealing to the properties of the filter function $\phi(t ; \alpha)$ it can be shown that 
$$\ell \in \text{Range}(Q^*) \iff  \liminf\limits_{\alpha \to 0} \langle x_\alpha \, , Ax_\alpha \rangle_{X\times X^*} < \infty .$$ 
This can be done by using the fact that 
\begin{align}\label{imagfunc}
\langle x_\alpha \, , Ax_\alpha \rangle_{X\times X^*}=\sum \frac{\phi^2(\lambda_n ; \alpha)}{\lambda_n} | \langle x_n,\ell  \rangle_{X\times X^*}|^2
\end{align}
along with some simple estimates of the above quantity using the properties of the filter function. Some common filter functions are given by 
\begin{align}\label{filters}
 \phi(t ; \alpha) =  \frac{t^2}{t^2+\alpha}, \,\, \,  \phi(t ; \alpha) = 1-\left( 1 - \beta t^2\right)^{1/\alpha} \,\, \textrm{ and } \,\,  \displaystyle{  \phi(t ; \alpha)= \left\{\begin{array}{lr} 1, &  t^2\geq \alpha,  \\
 				&  \\
 0,&  t^2 < \alpha
 \end{array} \right.}  
\end{align}
which corresponds to Tikhonov regularization, Landweber iteration (with  $\alpha=1/m$ for some $m \in \N$ and constant $\beta < 1/ \lambda^2_1$) and the Spectral cutoff respectively. It is clear that these filter functions satisfy the above constraints (see for e.g. \cite{IP-book}). 

{\color{black}Note, that the classical factorization method (i.e. without regularization) is given by using \eqref{imagfunc} with $\alpha=0$. Formally, this would imply that $\phi(t ; 0)=1$ and therefore one would be dividing by the singular values $\lambda_n$. This is not numerically stable since $\lambda_n \to 0$ as $n \to \infty.$  This will be seen in one of our numerical examples provided in a later section. }

Now, assume that the operator $A: X \to X^*$ also has the following factorization 
$$A=S^* T S  \quad \text{ where } \quad S: X \to V  \quad \text{ and } \quad T: V \to V^*$$
with $V$ also being a Hilbert space. Here the adjoint operator $S^*: V^* \to X^*$ is given by  
\begin{align}\label{adjoint}
\langle S x , v \rangle_{V\times V^*} = \langle x , S^* v \rangle_{X\times X^*} \quad \text{ for all } \quad  v \in {V} \textrm{ and } x \in X.
\end{align}
Furthermore, we assume $T$ is bounded and strictly coercive on Range$(S)$ i.e. 
$$\beta \| Sx \|^2_V \leq \langle Sx , TSx \rangle_{V\times V^*} \quad \text{ for all } x \in X.$$
If we assume that $S$ is a compact and injective then we have that $A: X \to X^*$ is positive and compact. From the previous discussion, this implies that $A=Q^* Q$ where $Q$ is the `square root' of the operator. Notice, that we have the estimate 
$$ \beta \| Sx \|^2_V  \leq \| Qx \|^2_H = \langle x \, , Ax \rangle_{X\times X^*} = \langle Sx , TSx \rangle_{V\times V^*}  \leq \| T \| _{V\to V^*} \| Sx \|^2_V $$
for all $x \in X$ by appealing to the boundedness and coercivity of the operator $T$. We can conclude, by Theorem 1 of \cite{range-lemma} that Range$\big( Q^* \big)$ = Range$(S^*)$. Putting everything together we have the following result. 

\begin{theorem}\label{reg-range-lemma}
Let $A: X \to X^*$ have the factorization $A=S^* T S$ such that $S: X \to V$ and $T: V \to V^*$ are bounded linear operators where $X$ and $V$ are Hilbert spaces. Assume that $S$ is compact and injective as well as $T$ being strictly coercive on Range$(S)$.  Then we have that
$$\ell \in\text{Range} (S^*) \iff \liminf\limits_{\alpha \to 0} \langle x_\alpha \, , Ax_\alpha \rangle_{X\times X^*} < \infty$$
where $x_\alpha$ is the regularized solution given by \eqref{regsolu} to $Ax=\ell$. 
\end{theorem} 
\begin{proof}
For details of the proof see \cite{RegFMme}. 
\end{proof}

Notice, that the result in Theorem \ref{reg-range-lemma} can be reformulated using the spectral decomposition of $A$ such that 
\begin{align}\label{rangefunc}
\ell \in\text{Range} (S^*) \iff \liminf\limits_{\alpha \to 0} \sum \frac{\phi^2(\lambda_n ; \alpha)}{\lambda_n} | ( \ell_n,\ell  )_{X^*}|^2< \infty
\end{align}
where we have used \eqref{imagfunc}. Here $\phi(t ; \alpha)$ again denotes the filter function used to find the regularized solution to $Ax=\ell$. From this we have related the range of $S^*$ to the spectral decomposition of $A$  just as in the traditional factorization method but we have a regularization step. This allows one to have a rigorous range characterization without having unstable numerical reconstructions due to the fact that $\lambda_n$ tend to zero rapidly. In the preceding section we will see how to apply Theorem \ref{reg-range-lemma} to inverse shape problems coming from inverse scattering. We note that this method has been used (without proof) for numerical examples in \cite{fm-gbc,Liem} for recovering scatterers with the classical factorization method. 

\section{Applications to Inverse Scattering}\label{InvScat}
In this section, we will see how the theory developed in section \ref{RFM} can be applied to solving inverse shape problems. Here we are interested in problems coming from the area of inverse scattering. This comes up in many areas of engineering and medical imaging. The goal is to recover the shape of an object using the measured scattering data with little to no a priori information about the object. The scatterer will be illuminated by an incident wave and we will show how to recover the scatterer from the measured scattering data using the regularized factorization method.  

\subsection{An example with far-field measurements}\label{ff}
We will now consider the inverse shape problem of reconstructing an isotropic scatterer using far-field measurements. This problem has been studied by {\color{black}many} researchers with many interesting reconstruction methods for e.g. \cite{BaoLi,GLSM,FMiso}. Here we will apply Theorem \ref{reg-range-lemma} to solve the inverse shape problem as well as {\color{black}provide} some numerical examples. The classical factorization method was studied for this problem in \cite{FMiso}. 

To begin, let $D \subset \R^m$ (for $m=2$ or 3) denote the unknown inhomogeneous isotropic scattering region with Lipschitz boundary. Here we assume that the incident plane wave given by $u^i(x)=\text{e}^{\text{i}kx\cdot d}$ is used to illuminate the scatterer where $d \in \mathbb{S}=\text{unit sphere/circle}$. The parameter $k>0$ denotes that wave number. The incident plane wave's interaction with the scatterer $D$ results in the radiating scattered field $u^s(x,d)$ that satisfies 
\begin{align}
\Delta u^s+k^2(1+q)u^s &= -k^2 q u^i  \, \,  \textrm{ in } \, \, \R^m \label{scatter1} \\
{\partial _r u^s} -\text{i}k u^s &=\mathcal{O}\big( { r^{-(m+1)/2} } \big) \, \,  \textrm{ as } \, \, r \rightarrow \infty.\label{src}
\end{align}
Here \eqref{src} is the Sommerfeld radiation condition and is assumed to hold uniformly with respect to the angular direction(s) with $r=|x|$. The contrast $q \in L^{\infty}(\R^m)$ defines the deviation in the refractive index from the background. Therefore, we let $n(x)=1+q(x)$ denote the refractive index which is the material parameter with the presence of the scatterer such that supp$(q)=D$.

Assuming that there is a constant $q_{\text{min}}$ where 
$$ \Re (q)\geq q_{\text{min}}>0 \quad \text{ and } \quad \Im (q) \geq 0 \quad \text{ for a.e. } \,\, x \in D$$
then the analysis in chapter 8 of \cite{Colto2013} implies that \eqref{scatter1}--\eqref{src} has a unique solution $u^s \in H^1_{loc}(\R^m)$. Since $u^s$ is a radiating solution to Helmholtz equation in $\R^m \setminus \overline{D}$ we have the expansion
$$u^s(x,d)= \gamma \frac{\text{e}^{\text{i}kr}}{r^{(m-1)/2}} \left\{ u^{\infty}(\hat{x}, d ) + \mathcal{O} \left( \frac{1}{r} \right) \right\} \; \textrm{  as  } \;  r \to \infty$$
where $u^{\infty}(\hat{x}, d )$ denotes the corresponding {\it far-field pattern} for \eqref{scatter1}--\eqref{src}. This quantity depending on the incident direction $d$ and the measurement direction $\hat x=x/r$. The constant 
$$\gamma = \frac{ \mathrm{e}^{\mathrm{i}\pi/4} }{ \sqrt{8 \pi k} } \,\,\, \text{in} \,\,\, \R^2 \quad \text{and} \quad  \gamma = \frac{1}{ 4\pi } \,\,\, \text{in} \,\,\, \R^3. $$
For this model we will assume that the far-field pattern is measured from the scattered field far away from the scatterer $D$. This implies that we have access to the measured far-field operator  
\begin{align}
F:L^2(\mathbb{S}) \longrightarrow  L^2(\mathbb{S}) \quad \text{ such that } \quad (F g)(\hat{x})=\int_{\mathbb{S}} u^{\infty}(\hat{x}, d ) g(d) \, \text{d}s(d) . \label{fo}
\end{align}
In order to solve the inverse shape problem of recovering $D$ from the knowledge of $F$ we will appeal to Theorem \ref{reg-range-lemma} along with the factorization analysis in \cite{FMiso}.

We now derive and use the factorization the far-field operator $F$ to solve the inverse problem. For this, motivated by \eqref{scatter1}--\eqref{src} we consider the problem 
$$ \Delta w+k^2(1+q)w = -k^2 q f  \quad  \textrm{ in } \quad \R^m$$ 
along with \eqref{src} for any $f \in L^2(D)$. 
Therefore, we note that $w  \in H^1_{loc}(\R^m)$ satisfies (see for e.g. \cite{Colto2013})
\begin{align}
w(x) = k^2 \int_D q(y)\Phi(x,y) \big[ w(y) + f(y) \big] \, \text{d}y  \quad \text{ for any } \quad x \in \R^m.\label{intequ1}
\end{align}
Here $\Phi$ denotes the fundamental solution for Helmholtz equation given by 
\begin{align}\label{fund-solu}
\Phi(\cdot \,, y) = \left\{ 
		\begin{array}{cl}
			\frac{\text{i}}{4} H^{(1)}_{0}(k|\cdot - y|) &\quad \text{for} \quad m = 2, \\[1.5ex]
			\displaystyle \quad\quad  \frac{\text{e}^{\text{i}k|\cdot - y|}}{4\pi|\cdot - y|} &\quad \text{for} \quad m = 3
		\end{array}
	\right.
\end{align}
where $H^{(1)}_{0}$ is the first kind Hankel function of order zero. Using the fact that 
$$\Phi(x,y) =  \gamma \frac{\text{e}^{\text{i}k|x|}}{|x|^{(m-1)/2}} \left\{ \text{e}^{-\text{i} k \hat{x} \cdot y }+ \mathcal{O} \left( \frac{1}{|x|} \right) \right\} \; \textrm{  as  } \;  |x| \to \infty$$
from \eqref{intequ1} we can conclude that the far-field pattern for $w$ is given by  
\begin{align}
w^\infty (\hat{x} )  = k^2 \int_D q(y) \text{e}^{-\text{i} k \hat{x} \cdot y } \big[ w(y) + f(y) \big] \, \text{d}y. \label{ff-pattern}
\end{align}
From this, we define the operator 
$$H: L^2( \mathbb{S}) \longrightarrow L^2(D) \quad \text{ such that } \quad  H g = \int_{\mathbb{S}} \text{e}^{\text{i}k  y \cdot d} g(d)  \mathrm{d} s(d)\Big|_{ D}$$
and {\color{black}its} adjoint 
$$ H^*: L^2( D) \longrightarrow L^2(\mathbb{S}) \quad \text{ such that } \quad H^* \varphi=  \int_{D} \text{e}^{-\text{i}k \hat{x}\cdot y} \varphi(y)  \mathrm{d}{y}$$
for any $g \in L^2(\mathbb{S})$ and $\varphi \in L^2(D)$. Lastly, we define the bounded linear operator $T: L^2(D) \to L^2(D)$ such that 
$$T f = k^2q[w +f] \big|_D\quad \text{ for any } \quad f \in L^2(D).$$
It is well-known that $F$ corresponds to the far-field pattern when $u^i$ is replaced by $Hg$. The representation \eqref{ff-pattern} implies that 
$$ Fg = k^2\int_D q(y)  \text{e}^{-\text{i} k \hat{x} \cdot y } \big[ w_g(y) + (Hg)(y) \big] \, \text{d}y$$
where $w_g$ is the scattered field for $f=Hg$. We now have the factorization 
$$F=H^* TH$$ 
where the operators $T$ and $H$ are as defined above.

Notice, that we have a symmetric  factorization of the far-field operator that is needed to apply the theory in section \ref{RFM}. We now need that the operators used in the factorization do indeed satisfy the assumptions of Theorem \ref{reg-range-lemma}. To this end, it is well known that $H$ is  is compact and injective as well as 
$$ \ell_z = \text{e}^{-\text{i}k \hat{x}\cdot z} \quad \text{ satisfies } \quad \ell_z \in \text{Range}(H^*) \iff z \in D.$$
The last piece of the puzzle is the coercivity of the middle operator. As it stands, the middle operator $T$ is not strictly coercive on the range of $H$. In order to solve this problem we consider the operator 
$$ F_{\sharp} = \big|  \Re(F) \big| +  \big| \Im (F) \big|.$$
where 
$$\Re(F) =\frac{1}{2} (F+F^*) \quad  \text{and} \quad \Im (F) = \frac{1}{2 \text{i}} (F-F^*).$$
Note, that $\Re(F)$ and $\Im (F)$ are self-adjoint compact operators by definition which implies that the absolute value can be compute via the spectral decomposition i.e. the Hilbert-Schmidt Theorem. From the analysis of the operator $T$ (see for e.g. chapter 4 \cite{kirschbook} for details) we have that 
$$F_{\sharp} = H^* T_{\sharp} H$$ 
where the new operator $T_{\sharp}$ is strictly coercive on $ L^2(D)$. Therefore, by appealing to Theorem \ref{reg-range-lemma} we have that 
\begin{align}\label{ff-solu}
z \in D \iff \liminf\limits_{\alpha \to 0} (g^{\alpha}_z , F_{\sharp} g^{\alpha}_z )_{L^2(\mathbb{S})} < \infty 
\end{align}
provided that $g^\alpha_z$ is the regularized solution to $F_{\sharp} g = \ell_z$. 

{\bf Numerical examples:} We now give some numerical reconstructions using \eqref{ff-solu} in two dimensions. To this end, we assume that the scatterer has small area. Then, we can exploit the Born approximation for the scattered field to simplify the calculations of the synthetic data. Therefore, we will compute the synthetic far-field data using the approximation 
$$u^{\infty}(\hat{x},d) \approx k^2 \int_{D}q(y) \mathrm{e}^{- \mathrm{i} k {y}  \cdot ( \hat{x} - d)} \, \mathrm{d}y.$$
In all of the preceding examples we will take a constant contrast in the scatterer as well as a fixed wave number given by $q=1+\text{i}$ and  $k=4$, respectively. 
We let the boundary of the scatterer to be given by  
$$\partial D = r(\theta) \left(\cos (\theta), \sin(\theta) \right) \quad \text{ for } \quad 0\leq \theta \leq 2 \pi. $$
Here the radial function $r(\theta)$ is given by either 
$$r(\theta) = 0.5 \left( |\sin(\theta)|^{10} +  0.1 |\cos(\theta)|^{10} \right)^{-1/10} \quad  \text{or} \quad r(\theta) = 0.5\big(1-0.25\sin(4\theta) \big)$$
for a rounded square shaped scatterer or star shaped scatterer, respectively. 

In Figure \ref{ff-square-recon}--\ref{ff-star-landweber-recon}, we plot the discretized version of the reciprocal to \eqref{ff-solu} in order to recover the scatterer. For this, we will  take a fixed regularization parameter $\alpha =10^{-6}$ in all our examples. 
Now, we need to define the discretized far-field operator with random noise added as 
$${\bf F}_{\delta} = \left[u^{\infty}(\hat{x}_i, d_j ) \left( 1 +\delta E_{i,j} \right) \right]_{i,j=1}^{64}$$
with random complex-valued matrix $\mathbf{E}$ satisfying $\| \mathbf{E} \|_2 =1$. 
Here, we take $\hat{x}_i, d_j$ to be equally spaced {\color{black}points} on the unit circle given by 
$$\hat{x}_i = d_i =(\cos \theta_i , \sin \theta_i) \quad  \text{with}\quad  \theta_i =2\pi(i-1)/64. $$
Therefore, following \cite{RegFMme} we have that the imaging functional that discretizes the reciprocal to \eqref{ff-solu} is given by (see also \eqref{imagfunc})
$$W(z)= \left[ \sum\limits_{j=1}^{64} \frac{\phi^2(\sigma_j  ; \alpha)}{\sigma_j} \big|({\bf u}_j , \boldsymbol{ \ell}_z)\big|^2 \right]^{-1} \,\, \text{ with } \,\,  \boldsymbol{\ell}_z = [\text{e}^{-\text{i}k \hat{x}_i \cdot z}]_{i=1}^{64}.$$
Here $\sigma_j$ are the singular values and ${\bf u}_j$ are the left singular vectors of 
$${\bf F}_{\delta , \sharp} = \big|  \Re({\bf F}_{\delta}) \big| +  \big| \Im ({\bf F}_{\delta}) \big|$$ 
and the filter function $\phi(t ; \alpha)$ is given by \eqref{filters}. To reiterate, the absolute value of a self-adjoint matrix is given by {\color{black}its} eigenvalue decomposition. 

By Theorem \ref{reg-range-lemma} and \eqref{ff-solu} we expect that $W(z)>0$ for $z \in D$ and $W(z)\approx 0$ for $z \notin D$. In the following examples for this section, we plot the imaging function $W(z)$ along with the true shape of the scatterer given by the dotted lines. 

\begin{figure}[tb]
\centering 
\includegraphics[scale=0.4]{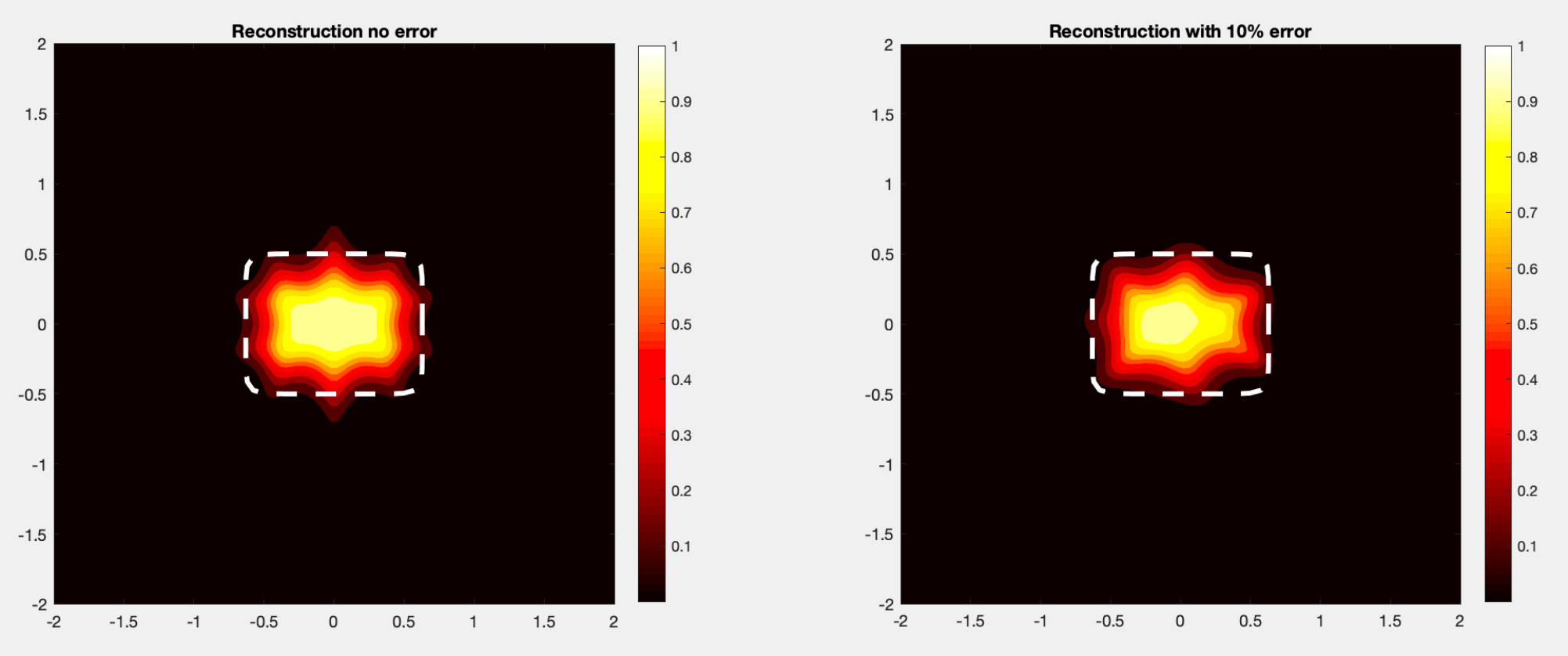}
\caption{Reconstruction of the rounded square shaped scatterer with the Tikhonov filter given in \eqref{filters}. Left: reconstruction with no added noise and Right: reconstruction with 10$\%$ added noise.}
\label{ff-square-recon}
\end{figure}

We will also check the influence of the filter function on the numerical reconstruction. The numerical examples in \cite{RegFMme} {\color{black}seem} to suggest that the reconstruction does not depend heavily on the regularization scheme used. We test that here where we present the {\color{black}reconstructed} star shaped scatterer with multiple filter functions. 

\begin{figure}[tb]
\centering 
\includegraphics[scale=0.4]{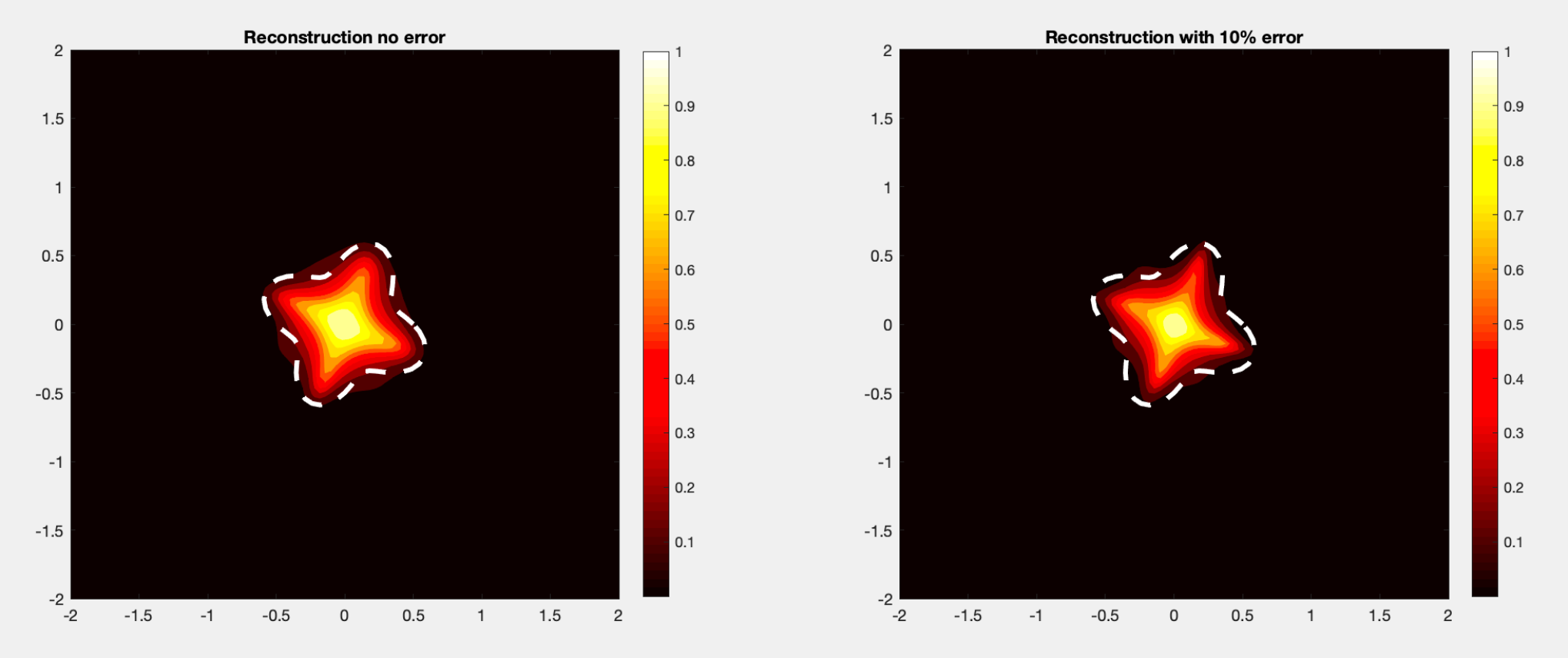}
\caption{Reconstruction of the star shaped scatterer with the Tikhonov filter given in \eqref{filters}. Left: reconstruction with no added noise and Right: reconstruction with 10$\%$ added noise.}
\label{ff-star-tik-recon}
\end{figure}

\begin{figure}[tb]
\centering 
\includegraphics[scale=0.4]{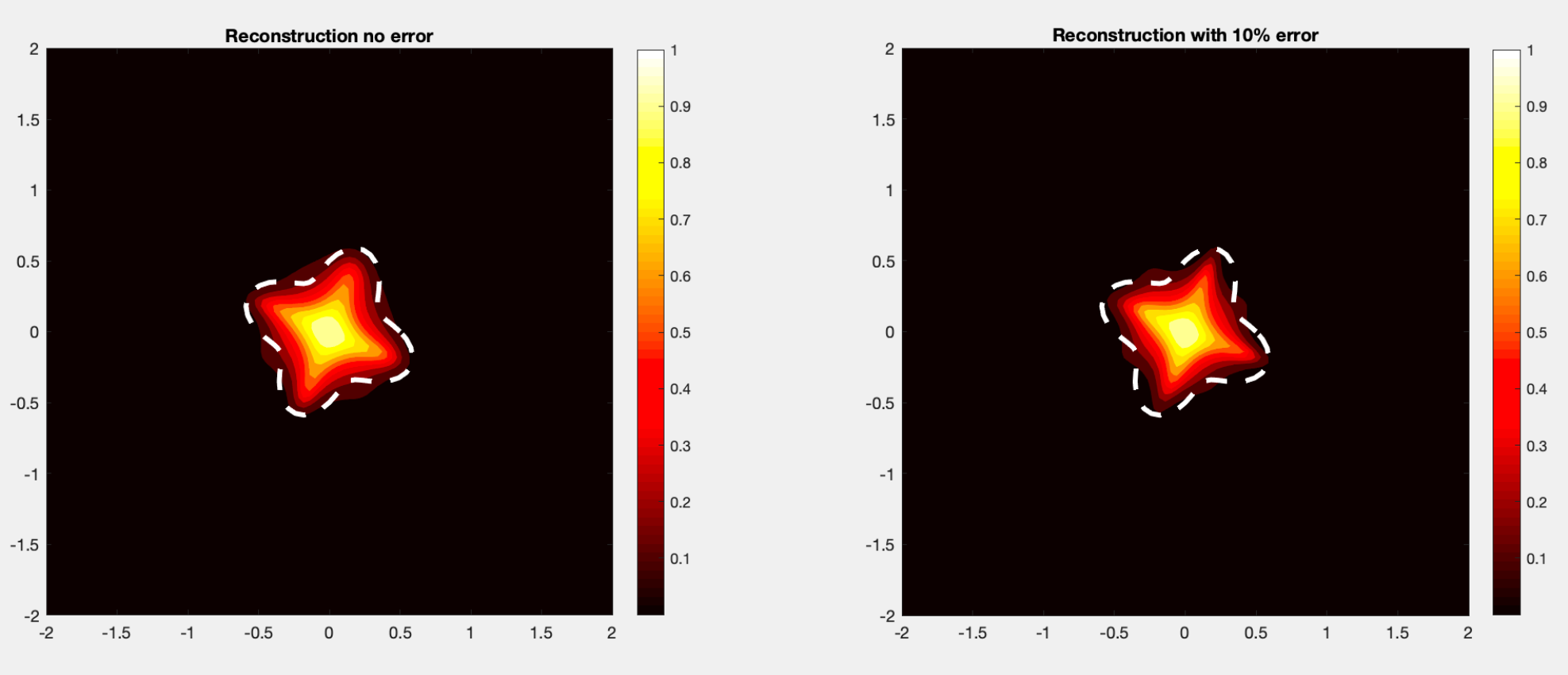}
\caption{Reconstruction of the star shaped scatterer with the Spectral cutoff filter given in \eqref{filters}. Left: reconstruction with no added noise and Right: reconstruction with 10$\%$ added noise.}
\label{ff-star-TSVD-recon}
\end{figure}

\begin{figure}[tb]
\centering 
\includegraphics[scale=0.4]{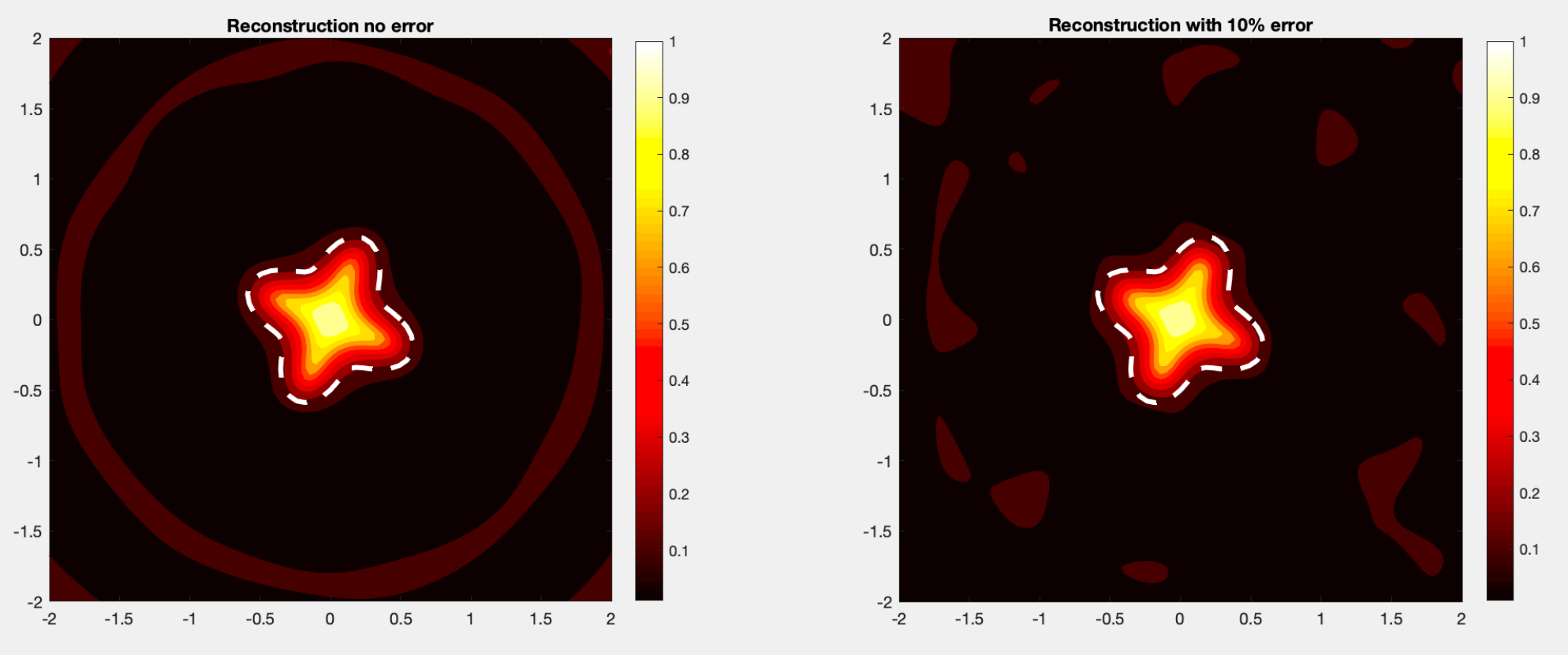}
\caption{Reconstruction of the star shaped scatterer with the Landweber filter given in \eqref{filters}. Left: reconstruction with no added noise and Right: reconstruction with 10$\%$ added noise.}
\label{ff-star-landweber-recon}
\end{figure}

In Figure \ref{ff-star-tik-recon}--\ref{ff-star-landweber-recon}, we reconstruct the star shaped scatterer with the three filter functions given in \eqref{filters}. As we can see, the choice of filter function {\color{black}seems} to cause little to no differences in the numerical reconstruction. 

{\color{black}Lastly, we wish to show that the regularization step is needed to insure stable reconstructions. In Figure \ref{ff-star-regcompair1}--\ref{ff-star-regcompair2}, we again recover the the star shaped scatterer. Here we plot the imaging functional $W(z)$ without regularization (i.e. $\alpha=0$ corresponding to the classical factorization method) and with regularization (i.e. $\alpha = 10^{-3}$). We give the reconstruction when no error is added to the data and 10$\%$ random noise is added to the data. As we can see, the case without regularization fails to recover the scatterer when noise is added to the far-field data. }
\begin{figure}[tb]
\centering 
\includegraphics[scale=0.4]{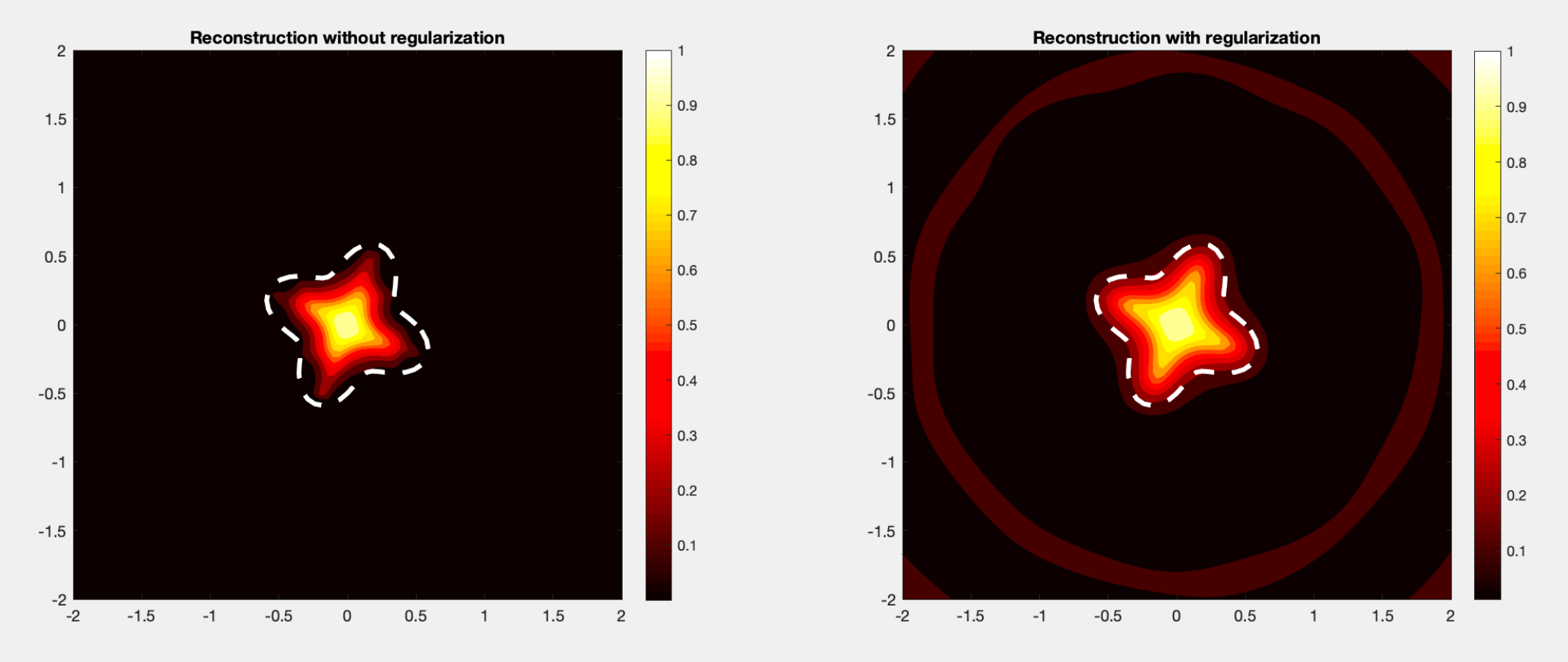}
\caption{Reconstruction of the star shaped scatterer with the Tikhonov filter given in \eqref{filters}. Left: reconstruction without regularization and Right: reconstruction with regularization. Here no error is added to the far-field data.}
\label{ff-star-regcompair1}
\end{figure}
 
\begin{figure}[tb]
\centering 
\includegraphics[scale=0.4]{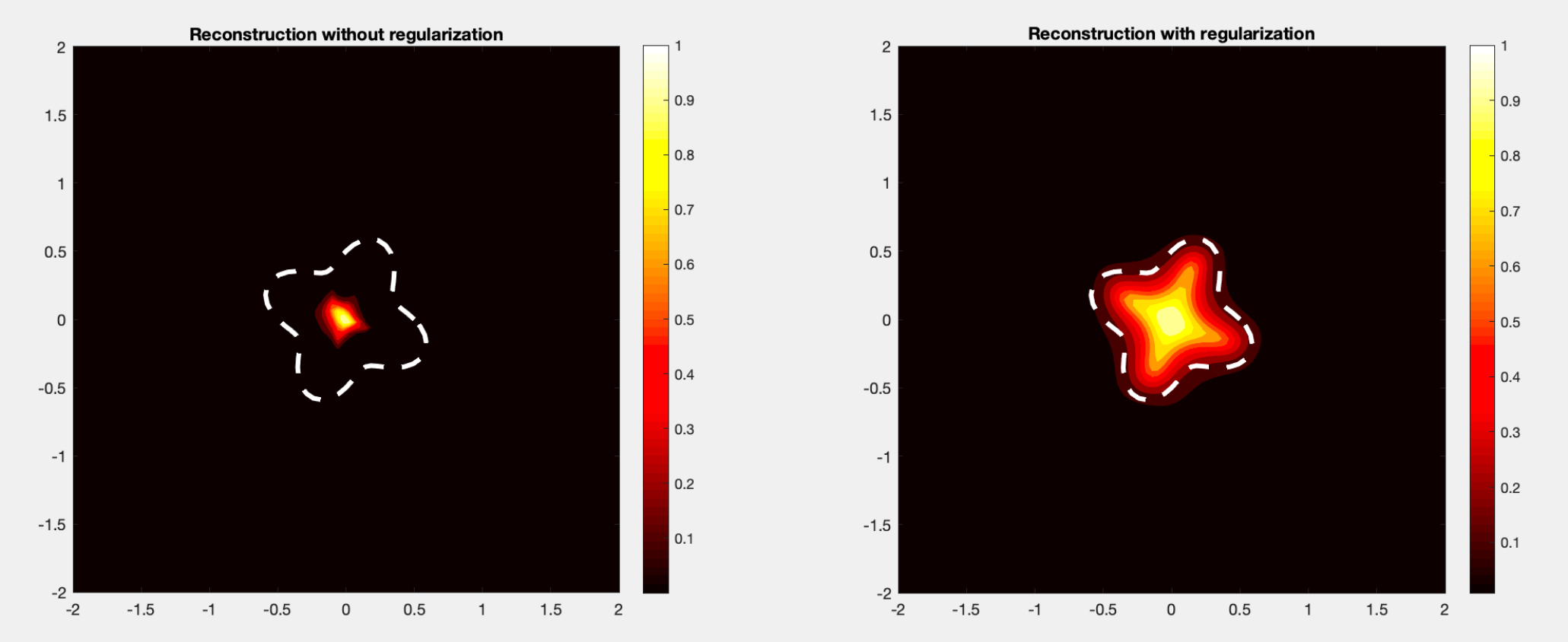}
\caption{Reconstruction of the star shaped scatterer with the Tikhonov filter given in \eqref{filters}. Left: reconstruction without regularization and Right: reconstruction with regularization. Here 10$\%$ error is added to the far-field data.}
\label{ff-star-regcompair2}
\end{figure}

\subsection{An example with near-field measurements}\label{nf}
We will now consider the inverse shape problem of reconstructing a sound soft scatterer using near-field measurements. One of the main difficulties when using a factorization method with near-field measurements is the fact that the near-field operator does not have a symmetric factorization as is needed to apply Theorem \ref{reg-range-lemma}. Due to this, researchers have developed analytical tools for post-processing the near-field measurements to give the corresponding operator a symmetric factorization. One way to achieve this is by using the Outgoing-to-Incoming operator. This was done in \cite{nf-fm-isotropic} for the classical factorization method using near-field measurements. In \cite{Harris-Rome} non-physical sources where used to insure that the near-field operator admits a symmetric factorization. Recently, in \cite{DSMnfme} a Dirichlet-to-Far-Field operator was used to convert the near-field measurements into far-field measurements for a direct sampling inversion method. This is advantageous due to the fact that the far-field operator usually admits a symmetric factorization as we have seen in the previous section.

Now, we will formulate the inverse scattering problem under consideration and then apply the Dirichlet-to-Far-Field operator to the measurements in order to apply Theorem \ref{reg-range-lemma}. To this end, we will assume that the unknown scatterer is denoted by $D \subset \R^m$ where $\partial D \in \mathcal{C}^2$ is a closed curve/surface such that $\R^m \setminus\overline{D}$ is connected. Again, we let $k>0$ denote the associated wave number. The scatterer is illuminated by a point source incident field $u^i( \cdot \, , y)=\Phi( \cdot \, , y)$ given by \eqref{fund-solu} where $y$ is the location of the point source on the curves/surface $\Gamma$. We will assume that the scatterer $D$ is contained in the region inclosed by the $\Gamma$ such that $\text{dist}(\Gamma , D)>0$. Therefore, the radiating scattered field $u^s( \cdot \, , y) \in H^1_{\text{loc}} (\R^m \setminus\overline{D})$ satisfies 
\begin{align}
\Delta u^s +k^2  u^s=0\,\,  \textrm{ in } \,\, \R^m \setminus\overline{D} \quad \textrm{ and } \quad u^s( \cdot \, , y) = -u^i( \cdot \, , y) \,\, \textrm{ on } \,\, \partial D \label{scalarprob} 
\end{align} 
 along with the radiation condition \eqref{src}. It is well known that for every $y \in \Gamma$ there is a unique scattered field. So we may assume that the scattered field $u^s(x,y)$ is known/measured for all $x,y \in \Gamma$. Therefore, we now define the so-called near-field operator 
$${N} : L^2(\Gamma) \longrightarrow L^2(\Gamma) \quad  \text{given by} \quad  ({N}g)(x) = \int_{\Gamma} u^s(x,y) g(y) \, \text{d}s(y).$$ 
Here in inverse shape problem is to recover $D$ from the knowledge of the near-field operator $N$. 

To this end, just as in the previous example we begin by deriving a factorization for the near-field operator. In order to continue, we make the assumption that $k^2$ is not a Dirichlet eigenvalue of the negative Laplacian in $D$. For this, motivated by \eqref{scalarprob} we consider the problem 
$$ \Delta w +k^2  w=0\,\,  \textrm{ in } \,\, \R^m \setminus\overline{D} \quad \textrm{ and } \quad w = -f \,\, \textrm{ on } \,\, \partial D $$ 
along with \eqref{src} for any $f \in H^{1/2}(\partial D)$. From the analysis done in \cite{DSMnfme} we have that the solution $w \in H^1_{\text{loc}} (\R^m \setminus\overline{D})$ has the integral representation
\begin{align}
w( x ) = -  \int_{\partial D} \Phi( x  ,  \omega) \big[ S^{-1} f \big](\omega) \, \text{d}s(\omega)  \quad \text{ for any } \quad x \in \R^m \setminus\overline{D}  \label{us-rep}
\end{align}
where 
$$S:H^{-1/2}(\partial D) \longrightarrow H^{1/2}(\partial D) \quad \text{such that} \quad S \varphi =\int_{\partial D} \Phi( \cdot \, ,  \omega) \varphi (\omega) \, \text{d}s(\omega) \Big|_{\partial D}$$
for any $\varphi \in H^{-1/2}(\partial D)$. By the assumption on $k$ it is known that $S$ has a bounded inverse (chapter 1 in \cite{kirschbook}) which implies that \eqref{us-rep} is well defined.  Now, we define the bounded linear operator 
\begin{align}\label{sl-op1}
M : L^2(\Gamma)  \longrightarrow L^2(\partial D) \quad  \text{given by} \quad   M g = \int_{\Gamma}  \Phi( \cdot \, , y)  g(y) \, \text{d}s(y) \Big|_{\partial D} 
\end{align}
and the dual-operator 
\begin{align}\label{sl-op2}
M^{\top} : L^2(\partial D)  \longrightarrow L^2(\Gamma) \quad  \text{given by} \quad M^{\top} \varphi = \int_{\partial D}  \Phi( \omega , \cdot)  \varphi(\omega) \, \text{d}s(\omega) \Big|_{\Gamma}.
\end{align}
Notice, that the dual-operator is with respect to the {\it bilinear} $L^2$ dual-product $\langle  \cdot \, , \cdot  \rangle_{L^2}$ such that 
$$\langle  \varphi , Mg  \rangle_{L^2(\partial D)} =\langle  M^{\top} \varphi , g  \rangle_{L^2(\Gamma)} \quad \text{for all } \quad g\in L^2(\Gamma) \textrm{ and } \varphi \in L^2(\partial D).$$
For the mapping properties of the operators $M$ and $M^\top$ see chapter 6 in \cite{mclean}. It is well-known that the near-field operator is the trace on $\Gamma$ for the solution to \eqref{scalarprob} provided that the incident field $u^i( \cdot \, , y)$ is replaced by $M g$.  Therefore, by equation \eqref{us-rep} we have that 
$$Ng = -  \int_{\partial D} \Phi( x  ,  \omega) \big[ S^{-1} Mg \big](\omega) \, \text{d}s(\omega)\Big|_{\Gamma} \quad \text{for all } \quad g \in L^2(\Gamma).$$ 
By the definition of the operator $M$ and {\color{black}its}  dual-operator we can conclude that 
\begin{align}
N = - M^{\top} \, S^{-1} \, M. \label{N-fac}
\end{align}   
Notice, that \eqref{N-fac} is not a symmetric factorization as in Theorem \ref{reg-range-lemma} due to the transpose rather than the adjoint. This implies that we must continue our analysis of the near-field operator in order to continue. 

We could employ the so-called Outgoing-to-Incoming operator as in \cite{nf-fm-isotropic}. From this the classical factorization method was studied in in \cite{nf-fm-isotropic} for three types of scatterers using near-field measurements. More recently, in \cite{DSMnfme} it has been shown that the near-field data can be transformed into the far-field data for the corresponding problem. We will now, use {\color{black}the} analysis in \cite{DSMnfme} to augment the measurements to use the regularized factorization method for this problem. 

We now define the Dirichlet-to-Far-Field operator which is a main component of the analysis. Now, let 
$v\in H^1_{\text{loc}}(\R^m \setminus \overline{\text{Int}(\Gamma)})$ be the unique solution to
\begin{align}
\Delta v +k^2 v  =  0 \quad \text{in}  \quad \R^m \setminus \overline{\text{Int}(\Gamma)} \quad \text{ with } \quad v|_\Gamma = f \label{eq-ext1} 
\end{align} 
along with the radiation condition \eqref{src} for any $ f \in H^{1/2}(\Gamma)$. Then, just as in \cite{DSMnf} we can define Dirichlet-to-Far-Field operator given by  
\begin{align}
\mathcal{Q}: H^{1/2}(\Gamma) \longrightarrow L^2(\mathbb{S}) \quad \text{such that } \quad (\mathcal{Q} f)(\hat{x}) = v^{\infty}(\hat{x}), \quad \forall \; \hat{x} \in \mathbb{S}. \label{Q-operator}
\end{align}
Now, we have that 
\begin{align}
(\mathcal{Q} \, M^{\top} \varphi) (\hat{x}) = \int_{\partial D} \text{e}^{-\text{i}k \hat{x}\cdot \omega} \varphi(\omega)   \text{d}s(\omega)  \quad  \textrm{for any } \quad \varphi \in L^2(\partial D) \label{qm-relation}
\end{align} 
by the asymptotic relations for the fundamental solution as $|x| \to \infty$.  Therefore, by \eqref{qm-relation} we define the bounded linear operator 
$$H: L^2( \mathbb{S}) \longrightarrow L^2( \partial D) \quad  \textrm{given by} \quad (H g)(\omega) = \int_{\mathbb{S}} \text{e}^{\text{i}k  \omega \cdot \hat{x}} g(\hat{x}) \mathrm{d} s(\hat{x}) \Big|_{\partial D} $$
for any $ g \in L^2(\mathbb{S})$. This corresponds to the trace of Herglotz wave function on the boundary of the scatterer. From this, we see that $\mathcal{Q} \, M^{\top} = H^*$. Now {\color{black}take} the transpose the expression to obtain that $(H^*)^\top = M \mathcal{Q}^\top$. We can now show that 
\begin{align*}
\left( (H^*)^\top g \right) (\omega) &= \int_{\mathbb{S}} \text{e}^{- \text{i}k  \omega \cdot \hat{x}} g(\hat{x}) \mathrm{d} s(\hat{x}) \Big|_{\partial D}\\
						   &= \int_{\mathbb{S}} \text{e}^{\text{i}k  \omega \cdot \hat{x}} g(-\hat{x}) \mathrm{d} s(\hat{x}) \Big|_{\partial D}\\
						   &=(H\mathcal{R} g)(\omega)
\end{align*}
where the operator 
$$\mathcal{R}:  L^2( \mathbb{S}) \longrightarrow L^2(\mathbb{S}) \quad \text{ is given by } \quad (\mathcal{R}g)(\hat{x}) = g(-\hat{x}).$$ 
Clearly, $\mathcal{R}$ is a bounded linear operator with $\mathcal{R} = \mathcal{R}^{-1}$. From this we can conclude that $H=M \mathcal{Q}^\top \mathcal{R}$. By the definition of $\mathcal{Q}$ and $\mathcal{R}$ we obtain that   
\begin{align} \label{FFT}
\mathcal{Q} N \mathcal{Q}^\top \mathcal{R} =- H^*S^{-1}H \quad \text{ where } \quad \mathcal{Q} N \mathcal{Q}^\top \mathcal{R} :  L^2( \mathbb{S}) \longrightarrow L^2(\mathbb{S})
\end{align}
by appealing to the factorization in \eqref{N-fac}. 

We can now relate the transformed operator $\mathcal{Q} N \mathcal{Q}^\top \mathcal{R}$ to the far-field operator for the scattering problem \eqref{scalarprob} where the incident field is given by a plane wave. Indeed, by \eqref{FFT} and equation (1.55) in \cite{kirschbook} we have that $\mathcal{Q} N \mathcal{Q}^\top \mathcal{R}=F$ where $F$ is the corresponding  far-field operator. This is important for our analysis here since $F$ does have a symmetric factorization. From Theorem 1.15 in \cite{kirschbook} we have that 
$$F=GS^*G \quad \text{ which implies that } \quad \mathcal{Q} N \mathcal{Q}^\top \mathcal{R} = GS^*G$$ 
where $S^*$ is the adjoint of $S$ defined above. The operator $G$ maps the trace on $\partial D$ to the far-field pattern for $w$ where 
$$ G: H^{1/2}( \partial D) \longrightarrow L^2(\mathbb{S}) \quad  \textrm{ is given by} \quad G w\big|_{\partial D} = w^\infty.$$ 
Recall, $w \in H^1_{\text{loc}} (\R^m \setminus\overline{D})$ can be written using \eqref{us-rep} and satisfies \eqref{src}. Just as in the previous section we consider 
$$ \big(\mathcal{Q} N \mathcal{Q}^\top \mathcal{R}\big)_{\sharp} = \big|  \Re\big(\mathcal{Q} N \mathcal{Q}^\top \mathcal{R}\big) \big| +  \big| \Im \big(\mathcal{Q} N \mathcal{Q}^\top \mathcal{R}\big) \big|.$$
By appealing to Lemma 1.14 of \cite{kirschbook} we again can conclude that we have the factorization 
$$  \big(\mathcal{Q} N \mathcal{Q}^\top \mathcal{R}\big)_{\sharp} = GS^*_\sharp G^*$$
where $S^*_\sharp$ is a strictly coercive operator. The last piece we need to complete the puzzle is the fact  that $G^*$ is compact and injective (see Theorem 1.15 of \cite{kirschbook}) along with 
$$ \ell_z = \text{e}^{-\text{i}k \hat{x}\cdot z} \quad \text{ satisfies } \quad \ell_z \in \text{Range}(G) \iff z \in D.$$
Therefore, by appealing to Theorem \ref{reg-range-lemma} we have that 
\begin{align}\label{nf-solu}
z \in D \iff \liminf\limits_{\alpha \to 0} \left(g^{\alpha}_z , \big(\mathcal{Q} N \mathcal{Q}^\top \mathcal{R}\big)_{\sharp} g^{\alpha}_z \right)_{L^2(\mathbb{S})} < \infty 
\end{align}
provided that $g^\alpha_z$ is the regularized solution to $\big(\mathcal{Q} N \mathcal{Q}^\top \mathcal{R}\big)_{\sharp} g = \ell_z$.

In order to apply \eqref{nf-solu} we need to compute the operators $\mathcal{Q}$ and $\mathcal{R}$. To due so, assume that $\Gamma = \partial B(0;\rho)$ for fixed $\rho>0$ in two dimensions then by {\color{black}appealing} to separation of variables and the asymptotic expansions for Hankel functions to obtain a formula for $\mathcal{Q}$. 
Indeed, we can use the fact that 
$$v(r,\theta) = \sum_{|n| = 0}^{\infty} \frac{{f}_n}{H^{(1)}_{n}(k \rho)}  H^{(1)}_{n}(kr) \text{e}^{\text{i}n\theta}  \quad \text{ for all} \quad n \in \mathbb{Z}$$
where $f_n$ are the Fourier coefficients for $f$ along with the asymptotic formula 
$$ H^{(1)}_{n}(kr) = \sqrt{\frac{2}{\pi k r}} \text{e}^{\text{i}kr-\text{i}n\pi/2 -\text{i}\pi/4 } + \mathcal{O}(r^{-3/2}) \quad \text{ as } \quad r \to \infty$$
to derive a computable formula for $\mathcal{Q}$.
From \cite{DSMnf} we have that the explicit formula 
$$ (\mathcal{Q} f)(\theta) = \int\limits_0^{2\pi} Q (\theta, \phi) f(\phi) \text{d}{\phi} \quad \text{where} \quad Q (\theta, \phi) = \frac{ (1-\text{i}) }{2 \pi \sqrt{\pi k}} \sum_{|n| = 0}^{\infty} \frac{\text{e}^{\text{i}n(\theta - \phi - \pi/2)}}{H^{(1)}_{n}(k\rho)}$$
with $f(\phi) =f\big( \rho (\cos\phi \, , \, \sin \phi) \big)$. Here, the constant radius $\rho$ is assumed to be large enough such that $D \subset B(0;\rho)$. When $\Gamma \neq \partial B(0;\rho)$ we can define Dirichlet-to-Far-Field operator $\mathcal{Q}$ by using boundary integral equations. See Section 2 of \cite{kirschbook} for a detailed construction. Now we need an explicit formula for the operator $\mathcal{R}$. To this end, we can use the expression given in \cite{DSMnfme} to write $\mathcal{R}$ as an integral operator with explicit kernel function. This expression uses the fact that  $\hat x =(\cos\theta \, , \, \sin \theta)$ for $\theta \in [0, 2\pi)$ in two dimensions. 
Now, by appealing to sum of angles formula we obtain
$$-\cos(\theta)=\cos(\theta+\pi) \quad \text{and} \quad -\sin(\theta)=\sin(\theta+\pi)$$
which implies that 
$$(\mathcal{R}g)(\theta) = g(\theta+\pi).$$ 
This formula uses the notation that $g(\theta) =g\big( (\cos\theta \, , \, \sin \theta) \big)$.
Therefore, {\color{black}by} using the Fourier series representation for $g$ we can conclude that 
$$ (\mathcal{R} g)(\theta) = \int\limits_0^{2\pi} R (\theta, \phi) g(\phi) \text{d}{\phi} \quad \text{where} \quad R(\theta, \phi) = \frac{1}{2\pi} \sum_{|n| = 0}^{\infty} \text{e}^{\text{i}n(\theta - \phi + \pi)}.$$
 For either $\mathcal{Q}$ or $\mathcal{R}$ we have that the series for the kernel function can be truncated to approximate the operators. It is shown in \cite{DSMnfme,DSMnf} that the truncated series is a valid approximation for both operators.  

{\bf Numerical examples:} Just as in the previous section, we will provide some numerical examples of \eqref{nf-solu} for recovering a sound soft scatterer. To this end, we again assume that the boundary of the scatterer is given by  
$$\partial D = r(\theta) \left(\cos (\theta), \sin(\theta) \right) \quad \text{ for } \quad 0\leq \theta \leq 2 \pi .$$
Here we take $r(\theta)$ to be given by either 
$$r(\theta) =0.25(2+0.5\cos(3\theta)) \quad  \text{or} \quad r(\theta) = 0.75\sqrt{0.75\cos^2(\theta)+0.07\sin^2(\theta)}$$
for an acorn shaped scatterer or peanut shaped scatterer, respectively. In all our examples we take $\Gamma=\partial B(0;5)$ (i.e. the disk with radius=5) and the wave number $k=4$. The location of the sources and receivers will be given by 
$${x}_i = y_i =5(\cos \theta_i , \sin \theta_i) \quad  \text{with}\quad  \theta_i =2\pi(i-1)/64.$$
This corresponds to 64 equally spaced points on the disk.

Now, we need to compute the scattering data $u^s(x_i,y_j)$ by solving \eqref{scalarprob}. Therefore, we use the fact that the scattered field is given by the series expansion 
$$u^s(x,y_j) = \sum\limits_{|n|=0}^{\infty} c_n(y_j) {H}^{(1)}_n (k|x|) \text{e}^{\text{i}n\theta_x} \quad \text{ for each } \quad j = 1, \cdots , 64.$$
The above representation is given by using separation of variables in $\R^2 \setminus\overline{D}$ for the Helmholtz equation. Notice, that the radiation condition \eqref{src} is satisfied by the asymptotic formula for the Hankel functions $ {H}^{(1)}_n (k|x|)$ as $|x| \to \infty$. 
Just as in \cite{DSMnfme} we will truncate the above series representation for $|n|=0,\cdots,15$ and solve for the series coefficients $c_n(y_j)$ such that 
$$ u^s(\widetilde{x},y_j) = -\Phi(\widetilde{x},y_j)  \quad \text{ for all } \quad \tilde{x} \in \partial D$$
and for each $ j = 1, \cdots , 64$. This is done by insuring that the above equality holds for each $\widetilde{x}_i =  r(\theta_i) \left(\cos (\theta_i), \sin(\theta_i) \right)$ for each $ i = 1, \cdots , 64$. So we solve the resulting $64 \times 31$ linear system of equations for each series coefficient $c_n(y_j)$. 
Once we have solved for the coefficients we have that the approximate scattering data on the measurement curve $\Gamma$ is given by 
$$u^s(x_i,y_j) \approx \sum\limits_{|n|=0}^{15} c_n(y_j) \text{H}^{(1)}_n (5k) \text{e}^{\text{i}n\theta_{i}}.$$
The discretized near-field operator with random noise added is given by  
$${\bf N}_{\delta} = \left[u^{s}({x}_i, y_j ) \left( 1 +\delta E_{i,j} \right) \right]_{i,j=1}^{64}$$
with random complex-valued matrix $\mathbf{E}$ satisfying $\| \mathbf{E} \|_2 =1$. 

Another piece we need in order to apply \eqref{nf-solu} is the discretization of the operators $\mathcal{Q}$ and $\mathcal{R}$. From the definition given earlier in this section we have that these operators can be written as integral operators with an explicit kernel given by an infinite series. To approximate the operators, we must first truncate the series representation for the kernel functions. Therefore, we now let 
$$\widetilde{Q} (\theta, \phi) = \frac{ (1-\text{i}) }{2 \pi \sqrt{\pi k}} \sum_{|n| = 0}^{10} \frac{\text{e}^{\text{i}n(\theta - \phi - \pi/2)}}{H^{(1)}_{n}(5k)}  \quad \text{and} \quad \widetilde{R}(\theta, \phi) = \frac{1}{2\pi} \sum_{|n| = 0}^{10} \text{e}^{\text{i}n(\theta - \phi + \pi)}.$$
This corresponds to the {\color{black}truncated series} for the kernel functions. 
We note that in \cite{DSMnfme,DSMnf} the approximate property of the {\color{black}truncated series} approximates was established. Using that  
$$(\mathcal{Q}f)(\theta) \approx  \int\limits_0^{2\pi} \widetilde{Q} (\theta, \phi) f(\phi) \text{d}{\phi} \quad \text{ and } \quad  (\mathcal{R}g)(\theta) \approx  \int\limits_0^{2\pi} \widetilde{R} (\theta, \phi) g(\phi) \text{d}{\phi}$$
we can employ a standard $64$ point Riemann sum collocation approximation for the integrals. From this we obtain a $64 \times 64$ discretization of the operators given by 
$$ {\bf Q} = \big[  \widetilde{Q} (\theta_{i} , \theta_{j}) \big]_{i,j = 1}^{64} \quad \text{and} \quad {\bf R} = \big[  \widetilde{R} (\theta_{i} , \theta_{j}) \big]_{i,j = 1}^{64}.$$
Again, we have taken $\theta_i =2\pi(i-1)/64$ for $i = 1, \cdots ,64$. 

Now that we have the discretized operators we define the discretized far-field transform of the {\color{black}near-field} operator as given by ${\bf Q} {\bf N}_\delta {\bf Q}^\top {\bf R}$. Therefore, we again let $\sigma_j$ be the singular values and ${\bf u}_j$ be the left singular vectors of the matrix
$$ \big({\bf Q} {\bf N}_\delta {\bf Q}^\top {\bf R} \big)_{\sharp} = \big|  \Re\big({\bf Q} {\bf N}_\delta {\bf Q}^\top {\bf R} \big) \big| +  \big| \Im \big({\bf Q} {\bf N}_\delta {\bf Q}^\top {\bf R} \big) \big|.$$ 
By \eqref{nf-solu} we have that the imaging functional 
$$W(z)= \left[ \sum\limits_{j=1}^{64} \frac{\phi^2(\sigma_j  ; \alpha)}{\sigma_j} \big|({\bf u}_j , \boldsymbol{ \ell}_z)\big|^2 \right]^{-1} \,\, \text{ with } \,\,  \boldsymbol{\ell}_z = [\text{e}^{-\text{i}k \hat{x}_i \cdot z}]_{i=1}^{64}$$
with filter function $\phi(t ; \alpha)$  given by \eqref{filters} can be used to approximate the scatterer $D$. Just as in the previous section we have that  $W(z)>0$ for $z \in D$ and $W(z)\approx 0$ for $z \notin D$ by appealing to Theorem \ref{reg-range-lemma}. In all our examples, we take the fixed regularization parameter $\alpha=10^{-6}$ and plot the imaging functional.

\begin{figure}[tb]
\centering 
\includegraphics[scale=0.4]{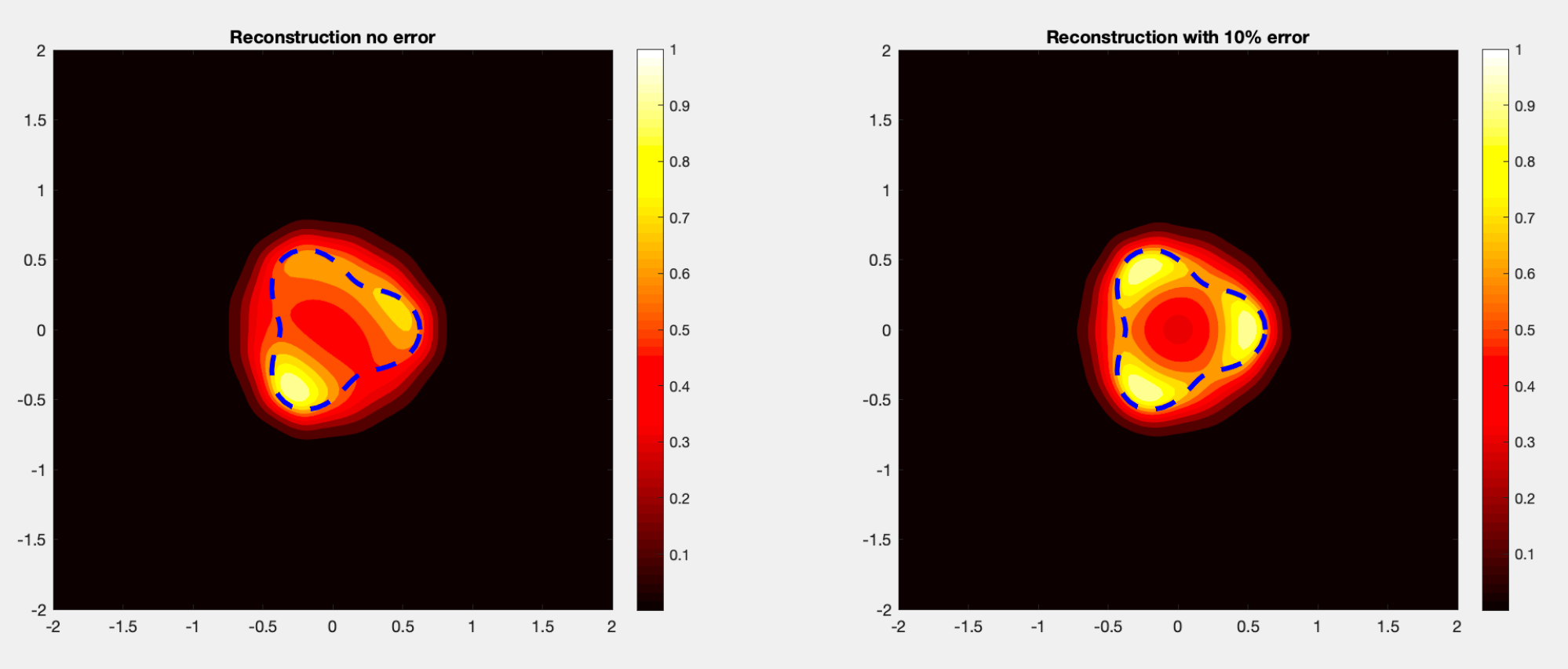}
\caption{Reconstruction of the ball shaped scatterer with the Tikhonov filter given in \eqref{filters}. Left: reconstruction with no added noise and Right: reconstruction with 10$\%$ added noise.}
\label{nf-acorn-recon}
\end{figure}

\begin{figure}[tb]
\centering 
\includegraphics[scale=0.4]{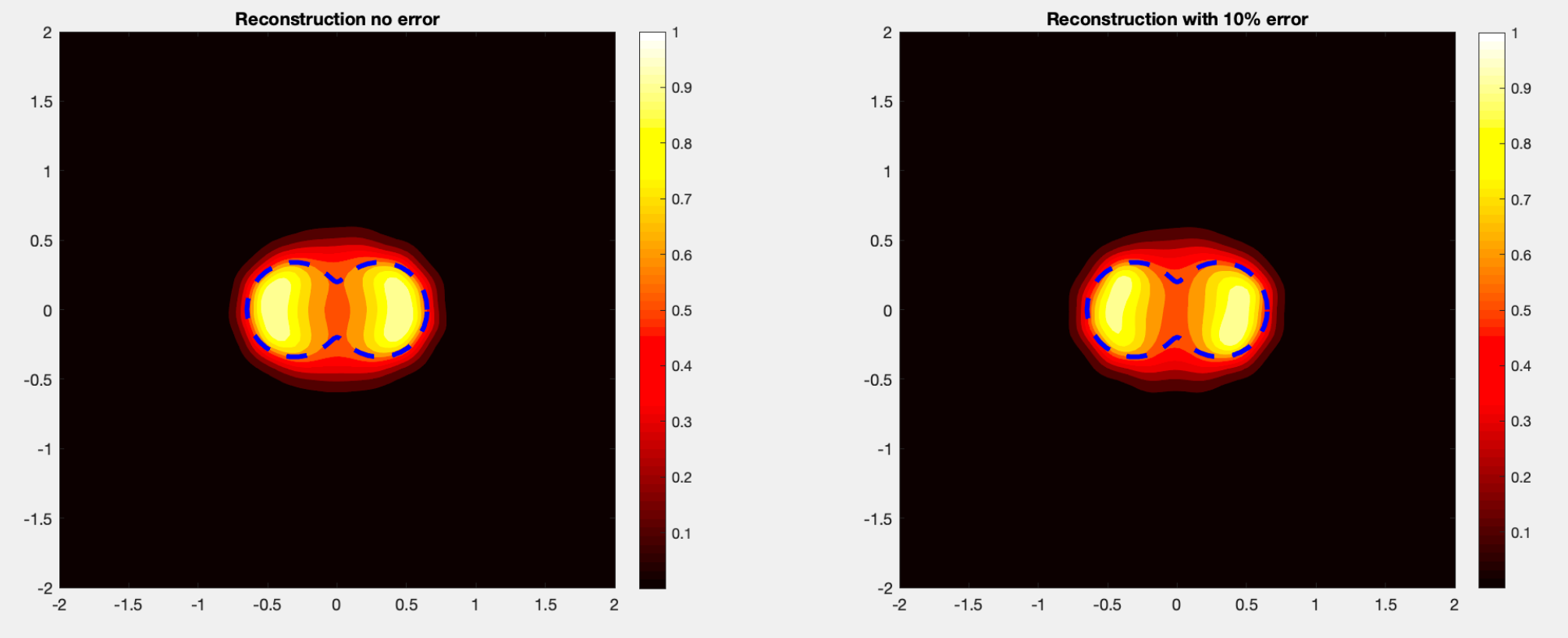}
\caption{Reconstruction of the peanut shaped scatterer with the Tikhonov filter given in \eqref{filters}. Left: reconstruction with no added noise and Right: reconstruction with 10$\%$ added noise.}
\label{nf-peanut-tik-recon}
\end{figure}

\begin{figure}[tb]
\centering 
\includegraphics[scale=0.4]{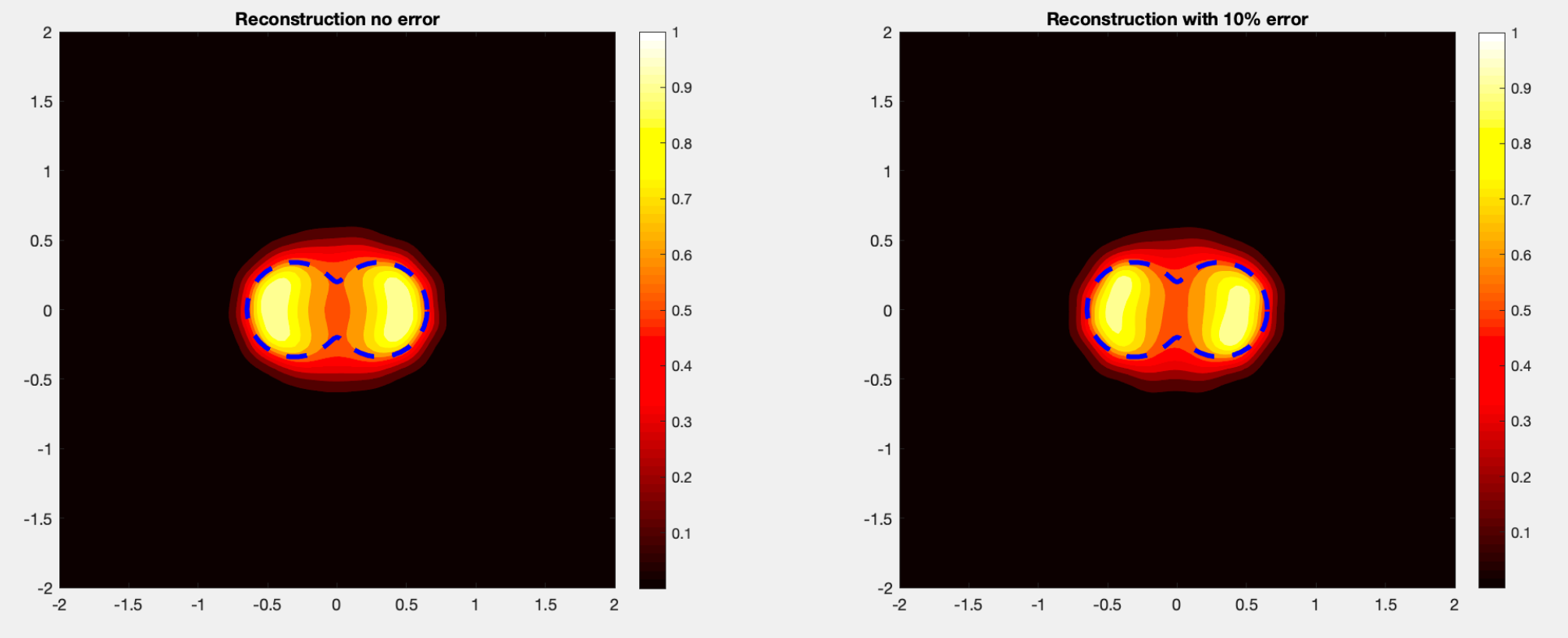}
\caption{Reconstruction of the peanut shaped scatterer with the Spectral cutoff filter given in \eqref{filters}. Left: reconstruction with no added noise and Right: reconstruction with 10$\%$ added noise.}
\label{nf-peanut-tik-recon}
\end{figure}

\begin{figure}[tb]
\centering 
\includegraphics[scale=0.4]{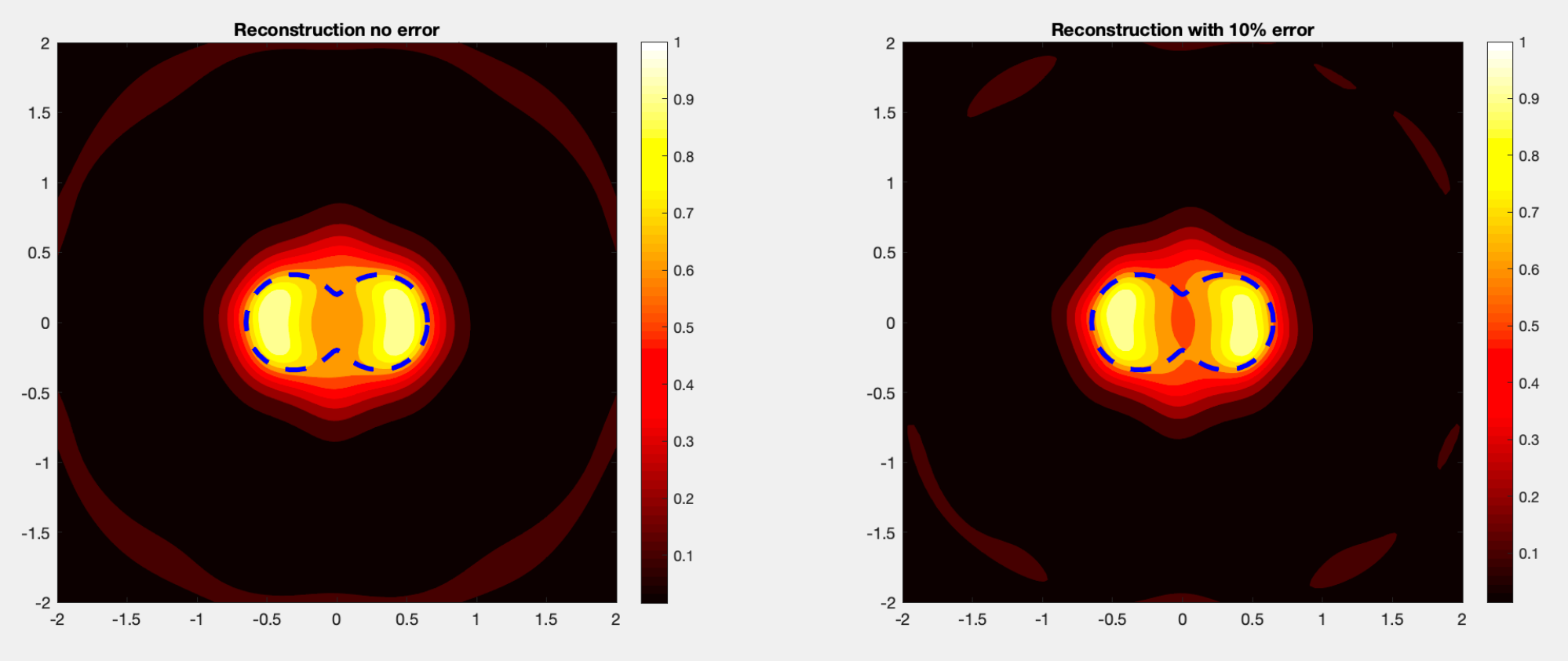}
\caption{Reconstruction of the peanut shaped scatterer with the Landweber filter given in \eqref{filters}. Left: reconstruction with no added noise and Right: reconstruction with 10$\%$ added noise.}
\label{nf-peanut-landweber-recon}
\end{figure}

As we see, in Figures \ref{nf-acorn-recon}--\ref{nf-peanut-landweber-recon} the filter function used in the reconstruction gives little to no difference computational results. Just as in the previous section, this implies that the reconstruction seems to not be sensitive with respect to the regularization scheme used in the imaging functional.

\section{Conclusions}\label{end}
In this paper, we have discussed a regularized version of the factorization method as well as its application to inverse scattering. Note, that this method was originally used in an application to diffuse optical tomography in \cite{RegFM}. We have seen that this method gives a new theoretically valid and analytically rigorous method for solving inverse shape problems. Moreover, we have applied this method to both near and far-field data sets. From our numerical investigation we see that the choice of regularization scheme seems to have little {\color{black}effect} on the reconstructions. A future direction of this research can be to provide theoretical justification of the regularized factorization method for a perturbed data operator. Also, the question of how to pick the regularization parameter is still open. In all our examples, we take the regularization parameter ad-hoc but one should determine a discrepancy principle to optimize the resolution of the imaging functional. {\color{black}From this, the main novelty of this paper is two fold. First, we have given a more extensive numerical study of this new qualitative reconstruction method. We have also, given another analytical and computation method for applying a factorization method for the near-field operator that lacks the symmetric factorization needed in the typical analysis.} Also, we can apply the regularized factorization method to other imaging modalities such as electrical impedance tomography.  \\

\noindent{\bf Acknowledgments:} The research of I. Harris is partially supported by the NSF DMS Grant 2107891.

\bibliographystyle{amsplain}

\begin{thebibliography}{99}
{

\bibitem{arens} 
\newblock T. Arens, 
\newblock Why linear sampling method works, 
\newblock {\it Inverse Problems} {\bf 20} 163--173 (2004).


\bibitem{arens2} 
\newblock T. Arens and A. Lechleiter,
\newblock Indicator Functions for Shape Reconstruction Related to the Linear Sampling Method,
\newblock {\it SIAM J. Imag. Sci.} {\bf 8:1} 513--535 (2015).



\bibitem{GLSM} 
\newblock L. Audibert and H. Haddar, 
\newblock { A generalized formulation of the linear sampling method with exact characterization of targets in terms of far-field measurements}, 
\newblock {\it Inverse Problems} {\bf 30}  035011 (2014).



%
%
%


\bibitem{BaoLi}
\newblock F. Cakoni,  D. Colton, and H. Haddar,
\newblock {Inverse medium scattering for the Helmholtz equation at fixed frequency}, 
\newblock {\it Inverse Problems}  {\bf 21}  1621 (2005).


\bibitem{TE-book}
\newblock F. Cakoni,  D. Colton, and H. Haddar,
\newblock {\it ``Inverse Scattering Theory and Transmission Eigenvalues''}, 
\newblock CBMS Series, SIAM Publications 88, (2016).



\bibitem{FM-wave}
\newblock F. Cakoni,  H. Haddar and A. Lechleiter, 
\newblock On the factorization method for a far field inverse scattering problem in the time domain, 
\newblock {\it SIAM J. Math. Anal.}, 2019, Vol. 51, No. 2: pp. 854--872.



\bibitem{fm-gbc} 
\newblock M. Chamaillard, N. Chaulet, and H. Haddar, 
\newblock Analysis of the factorization method for a general class of boundary conditions, 
\newblock {\it Journal of Inverse and Ill-posed Problems} {\bf 22} No. 5 643--670 (2014) .


\bibitem{Colto2013}
\newblock D. Colton and R. Kress.
\newblock {\em Inverse Acoustic and Electromagnetic Scattering Theory}.
\newblock Springer, New York, 3rd edition, 2013.





\bibitem{CK} 
\newblock D. Colton and A. Kirsch, 
\newblock {A simple method for solving inverse scattering problems in the resonance region}, 
\newblock {\it Inverse Problems} {\bf 12} 383--393  (1996).



\bibitem{range-lemma}
\newblock M. R. Embry, 
\newblock Factorization of operators on Banach space, 
\newblock {\it Proc. Amer. Math. Soc.}  {\bf 38} 587-590 (1973).






\bibitem{Gebauer}
\newblock B. Gebauer, 
\newblock The factorization method for real elliptic problems, 
\newblock {\it Z. Anal. Anwend.}, {\bf 25} 81--102 (2006).


\bibitem{FMheat}
\newblock J. Guo, G. Nakamura, and H. Wang,
\newblock The factorization method for recovering cavities in a heat conductor,
\newblock preprint (2019) arXiv:1912.11590

%


\bibitem{DSMnfme}
\newblock I. Harris, 
\newblock Direct methods for recovering sound soft scatterers from point source measurements.
\newblock {\it  Computation } {\bf 9(11)} 120 (2021).



\bibitem{RegFMme}
\newblock I. Harris, 
\newblock Regularization of the Factorization Method applied to diffuse optical tomography, 
\newblock {\it Inverse Problems}, {\bf 37}  125010 (2021).



\bibitem{DSMnf}
\newblock I. Harris D.-L. Nguyen and T.-P. Nguyen, 
\newblock Direct sampling methods for isotropic and anisotropic scatterers with point source measurements, 
\newblock preprint (2021) arXiv:2107.08138.



\bibitem{Harris-Rome}
\newblock I. Harris and S. Rome, 
\newblock Near field imaging of small isotropic and extended anisotropic scatterers, 
\newblock Applicable Analysis, {\bf 96:10} 1713--1736 (2017).




\bibitem{nf-fm-isotropic} G. Hu, J. Yang, B. Zhang and H. Zhang, 
\newblock {Near-field imaging of scattering obstacles with the factorization method}
\newblock {\it Inverse Problems} {\bf 30} 095005 (2014).


\bibitem{DOT-RtR}
\newblock N. Hyv\"{o}nen,
\newblock Application of a weaker formulation of the factorization method to the characterization of absorbing inclusions in optical tomography,
\newblock {\it Inverse Problems}, {\bf  21} 1331 (2005). 



\bibitem{RtR} 
\newblock N. Hyv\"{o}nen,
\newblock Characterizing inclusions in optical tomography,
\newblock {\it Inverse Problems}, {\bf  21} 737--751 (2004). 



\bibitem{IP-book}
\newblock A. Kirsch,
\newblock \emph{``An Introduction to the Mathematical Theory of Inverse Problems''},
\newblock  2$^{nd}$ edition Springer (New York) 2011.



\bibitem{firstFM}
\newblock A. Kirsch, 
\newblock Characterization of the shape of the scattering obstacle by the spectral data of the far field operator,
\newblock {\it Inverse Problems,} {\bf 14} 1489--512 (1998). 



\bibitem{FMiso}
\newblock A. Kirsch, 
\newblock The MUSIC-algorithm and the factorization method in inverse scattering theory for inhomogeneous media,
\newblock {\it Inverse Problems,} {\bf 18} 1025 (2002). 



\bibitem{kirschpp}
\newblock A. Kirsch, 
\newblock The Factorization Method for a Class of Inverse Elliptic Problems,
\newblock {\it Math. Nachrichten,} {\bf 278} 258--277 (2005).


\bibitem{kirschipbook}
\newblock A. Kirsch, 
\newblock  \emph{``An Introduction to the Mathematical Theory of Inverse Problems''},
\newblock 2$^{nd}$ edition Springer 2011.



\bibitem{kirschbook}
\newblock A. Kirsch and N.  Grinberg, 
\newblock \emph{``The Factorization Method for Inverse Problems''},
\newblock Oxford University Press, Oxford 2008.



\bibitem{RegFM} 
\newblock A. Lechleiter, 
\newblock A regularization technique for the factorization method, 
\newblock {\it Inverse Problems,} {\bf 22} 1605 (2006). 





\bibitem{mclean}
\newblock W. McLean,
\newblock {\it``Strongly elliptic systems and boundary integral equation''}. 
\newblock Cambridge University Press 2000.


\bibitem{Liem} 
\newblock D.-L. Nguyen,
\newblock Shape identification of anisotropic diffraction gratings for TM-polarized electromagnetic waves, 
\newblock {\it Applicable Analysis,} {\bf 93} 1458--1476 (2014).


}


\end{thebibliography}


\end{document}